\documentclass[a4j,12pt]{article}
\usepackage{amsmath}
\usepackage{amssymb}
\usepackage{amsthm}
\usepackage{enumerate}
\newtheorem{theorem}{Theorem}[section]
\newtheorem{Proposition}{Proposition}[section]
\newtheorem{Lemma}{Lemma}[section]
\newtheorem{Corollary}{Corollary}[section]
\newtheorem{Example}{Example}[section]

\begin{document}
\title{A determinant formula for relative congruence zeta functions for 
cyclotomic function fields}
\author{By D. Shiomi}
\date{}
\maketitle
%%%%%%%%%%%%%%%%%%%%%%%%%%%%%%%%%%%%%%%%%%%%%%%%%%%%%%%%%%%%%%%%%%%%%%%%%%%%%%%%%%%%%
%%
%%
%% Footnote
%%
%%
%%%%%%%%%%%%%%%%%%%%%%%%%%%%%%%%%%%%%%%%%%%%%%%%%%%%%%%%%%%%%%%%%%%%%%%%%%%%%%%%%
\footnote{2000 Mathematics Subject Classification: 11M38, 11R60}
\footnote{Key words and Phrases:  Congruence zeta function, Cyclotomic function fields}
%%%%%%%%%%%%%%%%%%%%%%%%%%%%%%%%%%%%%%%%%%%%%%%%%%%%%%%%%%%%%%%%%%%%%%%%%%%%
%%
%%
%%abstract
%%
%%
%%%%%%%%%%%%%%%%%%%%%%%%%%%%%%%%%%%%%%%%%%%%%%%%%%%%%%%%%%%%%%%%%%%%%%%%%%%
\begin{abstract}
In the paper [Ro], Rosen gave a determinant formula for relative class numbers for the $P$-th cyclotomic function fields
in the case of the monic irreducible polynomial $P$, which is regarded as an analogue of  the classical Maillet determinant.
In this paper,  we will give a determinant formula for the relative congruence zeta functions for cyclotomic function fields. 
Our formula is regarded as a generalization of the determinant formula for the relative class number.
\end{abstract}
%%%%%%%%%%%%%%%%%%%%%%%%%%%%%%%%%%%%%%%%%%%%%%%%%%%%%%%%%%%%%%%%%%%%%%%%%%%%
%%
%%
%%
%%
%%
%%
%% section 1
%%
%%
%%
%%
%%
%%%%%%%%%%%%%%%%%%%%%%%%%%%%%%%%%%%%%%%%%%%%%%%%%%%%%%%%%%%%%%%%%%%%%%%%%%%%
\section{Introduction}
Let $h_p^-$ be the relative class number of cyclotomic field of 
$p$-th root of unity.
In the paper [C-O], Carlitz and Olson computed
the number $h_p^-$ in terms of a certain classical
determinant, which is known as the Maillet determinant.

In the cyclotomic function field case, several authors gave an
analogue of Maillet determinants.

%%%%%%%%%%%%%%%%%%%%%%%%%%%%%%%%%%%%%%%%%%%%%%%%%%%%%%%%%%%%%%%%%%%%%%%%%%%%
%%
%%
%%%%%%%%%%%%%%%%%%%%%%%%%%%%%%%%%%%%%%%%%%%%%%%%%%%%%%%%%%%%%%%%%%%%%%%%%%%%%%%
Let $k$ be a field of rational functions
over a finite field $\mathbb{F}_q$ with $q$ elements.
Fix a generator $T$ of $k$, and let
$A=\mathbb{F}_q[T]$ be the polynomial subring of $k$.
Let $m$ be a monic polynomial of $A$, and $\Lambda_m$ be the
set of all of 
$m$-torsion points of the Carlitz module. 
The field $K_m$ obtained
by adding the points of $\Lambda_m$ to $k$ is
called the $m$-th cyclotomic function field.
For the definition of Carlitz module and basic facts of
cyclotomic function fields, see Section 2 below.
Let $K_m^+$ be   
the decomposition field of the infinite prime of $k$
in $K_m/k$, which is called the ``maximal real subfield'' in $K_m$.

%%%%%%%%%%%%%%%%%%%%%%%%%%%%%%%%%%%%%%%%%%%%%%%%%%%%%%%%%%%%%%%%%%%%%%%%%%%%
%%
%%
%%%%%%%%%%%%%%%%%%%%%%%%%%%%%%%%%%%%%%%%%%%%%%%%%%%%%%%%%%%%%%%%%%%%%%%%%%%%%%%
Let $h_m, h_m^+$ be orders of the divisor class group of degree $0$ 
for $K_m$, and $K_m^+$, respectively.
Define the relative class number $h_m^-$ of $K_m$ by
$h_m^-=h_m/h_m^+$.

%%%%%%%%%%%%%%%%%%%%%%%%%%%%%%%%%%%%%%%%%%%%%%%%%%%%%%%%%%%%%%%%%%%%%%%%%%%%
%%
%%
%%%%%%%%%%%%%%%%%%%%%%%%%%%%%%%%%%%%%%%%%%%%%%%%%%%%%%%%%%%%%%%%%%%%%%%%%%%%%%%
Rosen gave a determinant formula for $h_P^-$ in the case of 
the monic irreducible polynomial $P$
(cf. [Ro]), which is regarded as an analogue of
the Maillet determinant.
Recently, several authors generalized the Rosen's formula and
gave class number formulas. (cf. [B-K], [A-C-J]). 

%%%%%%%%%%%%%%%%%%%%%%%%%%%%%%%%%%%%%%%%%%%%%%%%%%%%%%%%%%%%%%%%%%%%%%%%%%%%
%%
%%
%%
%% In this paper
%%
%%
%%
%%%%%%%%%%%%%%%%%%%%%%%%%%%%%%%%%%%%%%%%%%%%%%%%%%%%%%%%%%%%%%%%%%%%%%%%%%%%%%%
Let $\zeta(s,K_m)$ be the congruence zeta function for $K_m$.
The function $\zeta(s,K_m)$ can be expressed by
\begin{eqnarray*}
\zeta(s,K_m)=\frac{P_m(q^{-s})}{(1-q^{-s})(1-q^{1-s})},
\end{eqnarray*}
where $P_m(X)$ is a polynomial with integral coefficients.
Then we have the decomposition $P_m(X)=P_m^{(+)}(X)P_m^{(-)}(X)$, 
where $P_m^{(+)}(X)$ is the polynomial corresponding to
the congruence zeta function $\zeta(s,K_m^+)$ for $K_m^+$.
On the polynomial $P_m^{(+)}(X)$, the author gave the
determinant formula in the paper [Sh].
We see that $P_m^{(-)}(q^{-s})=\zeta(s,K_m)/\zeta(s,K_m^+)$,
which is called the relative congruence zeta function for $K_m$.

%%%%%%%%%%%%%%%%%%%%%%%%%%%%%%%%%%%%%%%%%%%%%%%%%%%%%%%%%%%%%%%%%%%%%%%%%%%%
%%
%%
%%%%%%%%%%%%%%%%%%%%%%%%%%%%%%%%%%%%%%%%%%%%%%%%%%%%%%%%%%%%%%%%%%%%%%%%%%%%%%%
The main result of the present paper is to give the determinant formula
for $P_m^{(-)}(X)$.
Since $P_m^{(-)}(1)=h_m^-$, our formula is regarded as a generalization of 
the determinant formula for the relative class number.

%%%%%%%%%%%%%%%%%%%%%%%%%%%%%%%%%%%%%%%%%%%%%%%%%%%%%%%%%%%%%%%%%%%%%%%%%%%%
%%
%%
%%%%%%%%%%%%%%%%%%%%%%%%%%%%%%%%%%%%%%%%%%%%%%%%%%%%%%%%%%%%%%%%%%%%%%%%%%%%%%% 
As an application of our determinant formula, we will give an explicit
formula for some coefficients of low degree terms for $P_m^{(-)}(X)$.
%%%%%%%%%%%%%%%%%%%%%%%%%%%%%%%%%%%%%%%%%%%%%%%%%%%%%%%%%%%%%%%%%%%%%%%%%%%%
%%
%%
%%
%%
%%
%% section 2
%%
%%
%%
%%
%%
%%%%%%%%%%%%%%%%%%%%%%%%%%%%%%%%%%%%%%%%%%%%%%%%%%%%%%%%%%%%%%%%%%%%%%%%%%%%%
\section{Basic facts}
In this section, we will provide several 
basic facts
 of cyclotomic function fields and
its zeta functions. For the proof of these facts, see [G-R], [Ro 2], [Wa].
%%%%%%%%%%%%%%%%%%%%%%%%%%%%%%%%%%%%%%%%%%%%%%%%%%%%%%%%%%%%%%%%%%%%%%%%%%%%
%%
%%
%%
%%
%%
%% subsection 2.1
%%
%%
%%
%%
%%
%%%%%%%%%%%%%%%%%%%%%%%%%%%%%%%%%%%%%%%%%%%%%%%%%%%%%%%%%%%%%%%%%%%%%%%%%%%%%
\subsection{Cyclotomic function fields}
Let $K^{ac}$ be the algebraic closure of $k$. 
For $x \in K^{ac}$ and $m \in A$, 
we define the following action:
%%%%%%%%%%%%%%%%%%%%%%%%%%%%%%%%%%%%%%%%%%%%%%%%%%%%%%%%%%%%%%%%%%%%%%%%%%%%
%%
%% equation 1
%%
%%%%%%%%%%%%%%%%%%%%%%%%%%%%%%%%%%%%%%%%%%%%%%%%%%%%%%%%%%%%%%%%%%%%%
\begin{eqnarray}
m \cdot x= m(\varphi+\mu)(x),
\end{eqnarray}
where $\varphi,\;\mu$ are $\mathbb{F}_q$-linear maps of $K^{ac}$ defined by
%%%%%%%%%%%%%%%%%%%%%%%%%%%%%%%%%%%%%%%%%%%%%%%%%%%%%%%%%%%%%%%%%%%%%%%%%%%%
%%
%% equation  
%%
%%%%%%%%%%%%%%%%%%%%%%%%%%%%%%%%%%%%%%%%%%%%%%%%%%%%%%%%%%%%%%%%%%%%%
\begin{eqnarray*}
\varphi:K^{ac} \longrightarrow K^{ac}&& \;\;(x \mapsto x^q), \\
\mu:K^{ac} \longrightarrow K^{ac} &&\;\;(x\mapsto T \cdot x).
\end{eqnarray*}
By the above action, $K^{ac}$ becomes
a $A$-module, which is called the Carlitz module.
Let $\Lambda_m$ be the set of all $x$ satisfying  $m \cdot x=0$, which is a
cyclic sub-$A$-module of $K^{ac}$.
Fix a generator $\lambda_m$ of $\Lambda_m$. Then we have the following 
isomorphism of $A$-modules
%%%%%%%%%%%%%%%%%%%%%%%%%%%%%%%%%%%%%%%%%%%%%%%%%%%%%%%%%%%%%%%%%%%%%%%%%%%%
%%
%% equation (2)
%%
%%%%%%%%%%%%%%%%%%%%%%%%%%%%%%%%%%%%%%%%%%%%%%%%%%%%%%%%%%%%%%%%%%%%%
\begin{eqnarray}
A/(m) \longrightarrow \Lambda_m \; (a \mod m  \mapsto a \cdot \lambda_m),
\end{eqnarray}
where $(m)=mA$ is principal ideal generated by $m$.
Let  $(A/(m))^{\times}$ be the unit group of $A/(m)$, and
$\Phi (m)$ be the order of $(A/(m))^{\times}$.
Let $K_m$ be the field obtained by adding elements of $\Lambda_m$
to $k$. 
We call $K_m$ the $m$-th cyclotomic function field.
The extension $K_m/k$ is an abelian extension, and we get the
following isomorphism
%%%%%%%%%%%%%%%%%%%%%%%%%%%%%%%%%%%%%%%%%%%%%%%%%%%%%%%%%%%%%%%%%%%%%%%%%%%%
%%
%% equation (3)
%%
%%%%%%%%%%%%%%%%%%%%%%%%%%%%%%%%%%%%%%%%%%%%%%%%%%%%%%%%%%%%%%%%%%%%%
 \begin{eqnarray}
(A/(m))^{\times} \longrightarrow \text{Gal}(K_m/k)
 \; (a \; \text{mod} \; m \mapsto \sigma_{a \; \text{mod}\; m})
\end{eqnarray}  
where $\text{Gal}(K_m/k)$ is the Galois group of $K_m/k$, and
$\sigma_{a \; \text{mod} \; m}$ is the isomorphism
given by $\sigma_{a \; \text{mod} \; m}(\lambda_m)=a \cdot \lambda_m$.
%%%%%%%%%%%%%%%%%%%%%%%%%%%%%%%%%%%%%%%%%%%%%%%%%%%%%%%%%%%%%%%%%%%%%%
%%
%%%%%%%%%%%%%%%%%%%%%%%%%%%%%%%%%%%%%%%%%%%%%%%%%%%%%%%%%%%%%%%%%%%%%%
By using the above isomorphism, 
we find that the extension degree of $K_m/k$ is
$\Phi (m)$. \\
We see that $\mathbb{F}_q^{\times}$ is contained in
$(A/(m))^{\times}$. Let $K_m^{+}$ be the subfield of 
$K_m$ corresponding to $\mathbb{F}_q^{\times}$.
Again by 
the isomorphism (3), we find that 
the extension degree of $K_m^+/k$ is 
$\Phi (m)/(q-1)$.
Let $P_{\infty}$ be the unique prime of $k$ which corresponds 
to the valuation $v_{\infty}$ with  $v_{\infty}(T)<0$.
The prime $P_{\infty}$ splits completely in $K_m^+/k$, and
any prime of $K_m^+$ over $P_{\infty}$ is totally ramified in 
$K_m/K_m^+$.
Hence $K_m^+=K_m \cap k_{\infty}$ where $k_{\infty}$
 is the completion of $k$ by $v_{\infty}$.
The field $K_m^+$ is called the maximal real subfield of $K_m$,
which is an analogue of maximal 
real subfields of cyclotomic fields.\

%%%%%%%%%%%%%%%%%%%%%%%%%%%%%%%%%%%%%%%%%%%%%%%%%%%%%%%%%%%%%%%%%%%%%%%
%%
%%
%%
%% Dirichlet character
%%
%%
%%
%%%%%%%%%%%%%%%%%%%%%%%%%%%%%%%%%%%%%%%%%%%%%%%%%%%%%%%%%%%%%%%%%%%%%%
Next, we provide basic facts about Dirichlet characters.
For a monic polynomial $m \in A$, let
$X_m$ be the group of all primitive Dirichlet characters
of $(A/(m))^{\times}$.
Let $X_m^{+}$ be the set of characters contained in
$X_m$ such that $\chi(a)=1$ for any $a \in \mathbb{F}_q^{\times}$. 
%%%%%%%%%%%%%%%%%%%%%%%%%%%%%%%%%%%%%%%%%%%%%%%%%%%%%%%%%%%%%%%%%%%%%%%
%%
%% equation 4
%%
%%%%%%%%%%%%%%%%%%%%%%%%%%%%%%%%%%%%%%%%%%%%%%%%%%%%%%%%%%%%%%%%%%
Put
\begin{eqnarray}
\widetilde{K}=\bigcup_{m \text{:monic}}  K_m
\end{eqnarray}
where $m$ runs through all monic polynomials of $A$.
Let $\mathbb{D}$ be the group of  all
primitive Dirichlet characters. 
By the same argument as in Chapter $3$ in [Wa], 
we have a one-to-one correspondence between
finite subgroups of $\mathbb{D}$ and finite
subextension fields of $\widetilde{K}/k$.
The following theorem is useful to obtain the imformation
of primes.
%%%%%%%%%%%%%%%%%%%%%%%%%%%%%%%%%%%%%%%%%%%%%%%%%%%%%%%%%%%%%%%%%%%%%%
%%
%%
%%
%%
%% Theorem  2.1
%%
%%
%%
%%
%%%%%%%%%%%%%%%%%%%%%%%%%%%%%%%%%%%%%%%%%%%%%%%%%%%%%%%%%%%%%%%%%%
\begin{theorem} {\rm (cf.  [Wa], Theorem 3.7.)}
Let $X$ be a finite subgroup of $\mathbb{D}$, and
$K_X$ the associated field. For a irreducible monic polynomial $P \in A$,
put
%%%%%%%%%%%%%%%%%%%%%%%%%%%%%%%%%%%%%%%%%%%%%%%%%%%%%%%%%%%%%%%%%%%%%%
%%
%% equation  **
%%
%%%%%%%%%%%%%%%%%%%%%%%%%%%%%%%%%%%%%%%%%%%%%%%%%%%%%%%%%%%%%%%%%%
\begin{eqnarray*}
Y=\{ \chi \in X \; | \; \chi(P) \neq 0 \},\;\;
Z=\{ \chi \in X \; | \; \chi(P)=1 \}.
\end{eqnarray*}
Then, we have
%%%%%%%%%%%%%%%%%%%%%%%%%%%%%%%%%%%%%%%%%%%%%%%%%%%%%%%%%%%%%%%%%%%%%%
%%
%% equation  **
%%
%%%%%%%%%%%%%%%%%%%%%%%%%%%%%%%%%%%%%%%%%%%%%%%%%%%%%%%%%%%%%%%%
\begin{eqnarray*}
&&X/Y \simeq \text{ the inertia group of $P$ of $K_X/k$ },\\
&&Y/Z \simeq \text{ the cyclic group of order $f_P$},\\
&&X/Z \simeq \text{ the decomposition group of $P$ for $K_X/k$},
\end{eqnarray*}
where $f_P$ is the residue class degree of $P$ in $K_X/k$.
\end{theorem}
%%%%%%%%%%%%%%%%%%%%%%%%%%%%%%%%%%%%%%%%%%%%%%%%%%%%%%%%%%%%%%%%%%%%%%%%%%%%
%%
%%
%%
%%
%%
%% subsection 
%%
%%
%%
%%
%%
%%%%%%%%%%%%%%%%%%%%%%%%%%%%%%%%%%%%%%%%%%%%%%%%%%%%%%%%%%%%%%%%%%%%%%%%%%%
\subsection{The relative congruence zeta function}
Our next task is to investigate the congruence zeta function
for cyclotomic function fields.

Let $K$ be the geometric extension of $k$ of finite degree.
We define the congruence zeta function of $K$ by 
%%%%%%%%%%%%%%%%%%%%%%%%%%%%%%%%%%%%%%%%%%%%%%%%%%%%%%%%%%%%%%%%%%%%%%%%%%%%
%%
%% equation 5
%%
%%%%%%%%%%%%%%%%%%%%%%%%%%%%%%%%%%%%%%%%%%%%%%%%%%%%%%%%%%%%%%%%%%%%%%%%%%%
\begin{eqnarray}
\zeta(s,K)=\prod_{\mathcal{P}:\text{\rm prime}} \Bigl(1-\frac{1}
{{\mathcal{N} \mathcal{P}}^s}\Bigr)^{-1}
\end{eqnarray}
where $\mathcal{P}$ runs through all primes of $K$, 
and $\mathcal{N} \mathcal{P}$ is the number of elements of 
the reduce class field of a prime $\mathcal{P}$.
We see that $\zeta(s,K)$ converges absolutely for $\text{\rm Re}(s)>1$.
%%%%%%%%%%%%%%%%%%%%%%%%%%%%%%%%%%%%%%%%%%%%%%%%%%%%%%%%%%%%%%%%%%%%%%%%%%%%
%%
%%
%%
%%
%%
%% Theorem 2.2
%%
%%
%%
%%
%%
%%
%%%%%%%%%%%%%%%%%%%%%%%%%%%%%%%%%%%%%%%%%%%%%%%%%%%%%%%%%%%%%%%%%%%%%%%%%%%
\begin{theorem}
Let $g_K$ be the genus of $K$ and 
$h_K$ be
the order of divisor class group of degree $0$.
Then, there is a polynomial
$P_K(X) \in \mathbb{Z}[X]$ of degree $2g_K$ satisfying
%%%%%%%%%%%%%%%%%%%%%%%%%%%%%%%%%%%%%%%%%%%%%%%%%%%%%%%%%%%%%%%%%%%%%%%%%%%%
%%
%% equation 6
%%
%%%%%%%%%%%%%%%%%%%%%%%%%%%%%%%%%%%%%%%%%%%%%%%%%%%%%%%%%%%%%%%%%%%%%%%%%%%%
\begin{eqnarray}
\zeta(s,K)=\frac{P_K(q^{-s})}{(1-q^{-s})(1-q^{1-s})},
\end{eqnarray}
and $P_K(0)=1,\; P_K(1)=h_K$. 
\end{theorem}
Since the right-handside of equation (6) is meromorphic 
on the whole of $\mathbb{C}$, this equation 
provides the analytic continuation 
of $\zeta(s,K)$ to the whole of $\mathbb{C}$.

%%%%%%%%%%%%%%%%%%%%%%%%%%%%%%%%%%%%%%%%%%%%%%%%%%%%%%%%%%%%%%%%%%%%%%%%%%%%
%%
%%
%%
%% definition of S-zeta
%%
%%
%%%%%%%%%%%%%%%%%%%%%%%%%%%%%%%%%%%%%%%%%%%%%%%%%%%%%%%%%%%%%%%%%%%%%%%%%%%%
Next, we explain the zeta function of $\mathcal{O}_{K}$, which 
is the integral closure of $A$ in the field $K$.
We define the zeta function  
$\zeta(s,\mathcal{O}_K)$ for the ring $\mathcal{O}_K$ by
%%%%%%%%%%%%%%%%%%%%%%%%%%%%%%%%%%%%%%%%%%%%%%%%%%%%%%%%%%%%%%%%%%%%%%%%%%%%
%%
%% equation 7
%%
%%%%%%%%%%%%%%%%%%%%%%%%%%%%%%%%%%%%%%%%%%%%%%%%%%%%%%%%%%%%%%%%%%%%%%%%%%%%
\begin{eqnarray}
\zeta(s,\mathcal{O}_K)= \prod_{\mathcal{P}} 
\Bigl(1-\frac{1}{{\mathcal{N}\mathcal{P}}^s}\Bigr)^{-1}
\end{eqnarray}
where the product runs over all primes of $\mathcal{O}_K$.

%%%%%%%%%%%%%%%%%%%%%%%%%%%%%%%%%%%%%%%%%%%%%%%%%%%%%%%%%%%%%%%%%%%%%%%%%%%%
%%
%%
%%%%%%%%%%%%%%%%%%%%%%%%%%%%%%%%%%%%%%%%%%%%%%%%%%%%%%%%%%%%%%%%%%%%%%%%%%%%
Let X be a finite subgroup of $\mathcal{D}$,
and $K_X$ be the associated field. 
By the same argument as in the case of number
fields (cf. [Wa]), 
we have the following decomposition by $L$-functions
%%%%%%%%%%%%%%%%%%%%%%%%%%%%%%%%%%%%%%%%%%%%%%%%%%%%%%%%%%%%%%%%%%%%%%%%%%%%
%%
%% equation (8)
%%
%%%%%%%%%%%%%%%%%%%%%%%%%%%%%%%%%%%%%%%%%%%%%%%%%%%%%%%%%%%%%%%%%%%%%%%%%%%%
\begin{eqnarray}
\zeta(s,\mathcal{O}_{K_X})=\prod_{\chi \in X}L(s,\chi)
\end{eqnarray}
where the $L$-function is defined by 
$L(s,\chi)=\displaystyle\prod_P \Bigl(1-\frac{\chi(P)}{\mathcal{N}P^s}
\Bigr)^{-1}$
with $P$ running through all monic irreducible polynomials of $A$.

Let $f_{\infty},\;g_{\infty}$ be the residue class degree of $P_{\infty}$
in $K_X/k$, and the number of prime in $K_X$ over $P_{\infty}$, 
respectively.
Then we have
%%%%%%%%%%%%%%%%%%%%%%%%%%%%%%%%%%%%%%%%%%%%%%%%%%%%%%%%%%%%%%%%%%%%%%%%%%%%
%%
%% equation 9
%%
%%%%%%%%%%%%%%%%%%%%%%%%%%%%%%%%%%%%%%%%%%%%%%%%%%%%%%%%%%%%%%%%%%%%
\begin{eqnarray}
\zeta(s,K_X)= \zeta(s,\mathcal{O}_{K_X})(1-q^{-sf_{\infty}})^{-g_{\infty}}.
\end{eqnarray}
From now on, we will focus on cyclotomic function field case.
For a monic polynomial $m \in A$, let $K_m,\; K_m^+$ be the 
$m$-th cyclotomic function field and its maximal real subfield.
The relative congruence zeta function $\zeta^{(-)}(s,K_m)$
 is defined by
%%%%%%%%%%%%%%%%%%%%%%%%%%%%%%%%%%%%%%%%%%%%%%%%%%%%%%%%%%%%%%%%%%%%%%%%%%%%
%%
%% equation 10
%%
%%%%%%%%%%%%%%%%%%%%%%%%%%%%%%%%%%%%%%%%%%%%%%%%%%%%%%%%%%%%%%%%%%%%
\begin{eqnarray}
\zeta^{(-)}(s,K_m)=\frac{\zeta(s,K_m)}{\zeta(s,K_m^+)}.
\end{eqnarray}
By Theorem 2.2, there are polynomials $P_m(X),\;P_m^{(+)}(X)$ with
integral coefficients such that
%%%%%%%%%%%%%%%%%%%%%%%%%%%%%%%%%%%%%%%%%%%%%%%%%%%%%%%%%%%%%%%%%%%%%%%%%%%%
%%
%% equation 
%%
%%%%%%%%%%%%%%%%%%%%%%%%%%%%%%%%%%%%%%%%%%%%%%%%%%%%%%%%%%%%%%%%%%%%
\begin{eqnarray*}
\zeta(s,K_m)=\frac{P_m(q^{-s})}{(1-q^{-s})(1-q^{1-s})}, \\
\zeta(s,K_m^+)=\frac{P_m^{(+)}(q^{-s})}{(1-q^{-s})(1-q^{1-s})}. 
\end{eqnarray*}
Put $P_m^{(-)}(X)=P_m(X)/P_m^{(+)}(X)$, then we have
%%%%%%%%%%%%%%%%%%%%%%%%%%%%%%%%%%%%%%%%%%%%%%%%%%%%%%%%%%%%%%%%%%%%%%%%%%%%
%%
%% equation 11
%%
%%%%%%%%%%%%%%%%%%%%%%%%%%%%%%%%%%%%%%%%%%%%%%%%%%%%%%%%%%%%%%%%%%%%
\begin{eqnarray}
\zeta^{(-)}(s, K_m)=P_m^{(-)}(q^{-s}).
\end{eqnarray}
Notice that the fields $K_m,\;K_m^+$ associate to $X_m,\;X_m^+$, respectively.  
Since any prime in $K_m^+$ above $P_{\infty}$ is totally ramified in $K_m/K_m^+$,
we have
%%%%%%%%%%%%%%%%%%%%%%%%%%%%%%%%%%%%%%%%%%%%%%%%%%%%%%%%%%%%%%%%%%%%%%%%%%%%
%%
%% equation 12
%%
%%%%%%%%%%%%%%%%%%%%%%%%%%%%%%%%%%%%%%%%%%%%%%%%%%%%%%%%%%%%%%%%%%%%
\begin{eqnarray}
P_m^{(-)}(q^{-s}) = \prod_{\chi \in X_m^-}L(s,\chi)
\end{eqnarray}
where $X_m^-=X_m - X_m^+$. 

The $L$-function associated to the non-trivial
 character can
be expressed by the polynomial of $q^{-s}$ 
with complex coefficients. Hence 
we see that $P_m^{(-)}(X)$ is the polynomial with integral coefficients.
%%%%%%%%%%%%%%%%%%%%%%%%%%%%%%%%%%%%%%%%%%%%%%%%%%%%%%%%%%%%%%%%%%%%%%%%%%%%
%%
%%
%%
%%
%%
%% section 3
%%
%%
%%
%%
%%
%%%%%%%%%%%%%%%%%%%%%%%%%%%%%%%%%%%%%%%%%%%%%%%%%%%%%%%%%%%%%%%%%%%%%%%%%%%%%
\section{The determinant formula for $P_m^{(-)}(X)$}
In the previous section, we defined the relative congruence zeta 
function $\zeta^{(-)}(s,K_m)$ for the $m$-th cyclotomic function field,
and we showed that $\zeta(s,K_m)$ is expressed by the polynomial $P_m^{(-)}(X)$
with integral coefficients.
The goal of this section is to give a determinant formula for 
$P_m^{(-)}(X)$. First, we will prepare some notations to
construct the determinant formula.

Let $m$ be a monic polynominal of degree $d$. 
For $\alpha \in (A/(m))^{\times}$, there is a unique element of 
$r_{\alpha} \in A$ 
satisfying 
%%%%%%%%%%%%%%%%%%%%%%%%%%%%%%%%%%%%%%%%%%%%%%%%%%%%%%%%%%%%%%%%%%%%%%%%%%%%
%%
%% equation  **
%%
%%%%%%%%%%%%%%%%%%%%%%%%%%%%%%%%%%%%%%%%%%%%%%%%%%%%%%%%%%%%%%%%%%%%%%%%%%%
\begin{eqnarray*}
r_{\alpha}&=&a_n T^n+a_{n-1} T^{n-1}+\cdots +a_0 \;\;\;(n= \deg r_{\alpha}<d),\\
r_{\alpha}& \equiv& \alpha \mod m,
\end{eqnarray*}
where $\deg f$ denotes the degree of the polynomial $f$. Then we define
%%%%%%%%%%%%%%%%%%%%%%%%%%%%%%%%%%%%%%%%%%%%%%%%%%%%%%%%%%%%%%%%%%%%%%%%%%%%
%%
%% equation  **
%%
%%%%%%%%%%%%%%%%%%%%%%%%%%%%%%%%%%%%%%%%%%%%%%%%%%%%%%%%%%%%%%%%%%%%%%%%%%%
\begin{eqnarray*}
\text{Deg} (\alpha)=n,\;\;\;L(\alpha)=a_n \in \mathbb{F}_q^{\times}
\end{eqnarray*}
and $c^{\lambda}(\alpha)=\lambda^{-1}(L(\alpha))$ for the character
$\lambda$ of $\mathbb{F}_q^{\times}$.
Put $N_m=\Phi(m)/(q-1)$. Let $\alpha_1, \alpha_2,\cdots \alpha_{N_{m}}$
be all of the elements of $(A/(m))^{\times}$ with $L(\alpha)=1$,
which are the complete system of representatives for 
$\mathcal{R}_m=(A/(m))^{\times}/\mathbb{F}_q^{\times}$.
We put
%%%%%%%%%%%%%%%%%%%%%%%%%%%%%%%%%%%%%%%%%%%%%%%%%%%%%%%%%%%%%%%%%%%%%%%%%%%%
%%
%%
%%
%% equation   definition c_ij
%%
%%
%%%%%%%%%%%%%%%%%%%%%%%%%%%%%%%%%%%%%%%%%%%%%%%%%%%%%%%%%%%%%%%%%%%%%%%%%%%
\begin{eqnarray*}
c_{ij}^{\lambda}&=&c^{\lambda}(\alpha_i \alpha_j^{-1})
\;\;(i,j=1,2,..., N_m),\\
d_{ij}&=& \text{Deg}(\alpha_i \alpha_j^{-1})   \;\;(i,j=1,2,...,N_m).
\end{eqnarray*}
For any character $\lambda$ of $\mathbb{F}_q^{\times}$, we define the 
matrix
%%%%%%%%%%%%%%%%%%%%%%%%%%%%%%%%%%%%%%%%%%%%%%%%%%%%%%%%%%%%%%%%%%%%%%%%%%%%
%%
%% equation  
%%
%%%%%%%%%%%%%%%%%%%%%%%%%%%%%%%%%%%%%%%%%%%%%%%%%%%%%%%%%%%%%%%%%%%%%%%%%%%
\begin{eqnarray*}
D_{m}^{( \lambda )}(X)=(c_{ij}^{\lambda}X^{ d_{ij}})_{i,j=1,2,..., N_m}.
\end{eqnarray*}
%%%%%%%%%%%%%%%%%%%%%%%%%%%%%%%%%%%%%%%%%%%%%%%%%%%%%%%%%%%%%%%%%%%%%%%%%%%
%%
%%
%%
%%equation  13
%%
%%
%%
%%%%%%%%%%%%%%%%%%%%%%%%%%%%%%%%%%%%%%%%%%%%%%%%%%%%%%%%%%%%%%%%%%%%%%%%%%%
The following matrix plays an essential role
in our argument
\begin{eqnarray}
D_m^{(-)}(X)=\prod_{\lambda \neq 1} D_{m}^{(\lambda)}(X)
\end{eqnarray}
where the product runs over all non-trivial characters of
$\mathbb{F}_q^{\times}$.  
Notice that $d_{ij}>0$ in the case $i\neq j$, 
and $d_{ij}=0,\;c_{ij}^{\lambda}=1$ in the case $i=j$.
Thus 
$D_m^{(-)}(0)$ is the unit matrix.   
To state the  main result, we prepare the polynomial $J_m^{(-)}(X)$
defined by
%%%%%%%%%%%%%%%%%%%%%%%%%%%%%%%%%%%%%%%%%%%%%%%%%%%%%%%%%%%%%%%%%%%%%%%%%%%
%%
%% equation  **
%%
%%%%%%%%%%%%%%%%%%%%%%%%%%%%%%%%%%%%%%%%%%%%%%%%%%%%%%%%%%%%%%%%%%%%%%%%%%%
\begin{eqnarray}
J_m^{(-)}(X)=\prod_{\chi \in X_m^-}\prod_{Q | m}
(1-\chi (Q)X^{\deg Q})
\end{eqnarray}
where $Q$ is an irreducible monic polynomial dividing $m$. To begin with,
we prove the following proposition.
%%%%%%%%%%%%%%%%%%%%%%%%%%%%%%%%%%%%%%%%%%%%%%%%%%%%%%%%%%%%%%%%%%%%%%%%%%%
%%
%%
%% Proposition  3.1
%%
%%
%%%%%%%%%%%%%%%%%%%%%%%%%%%%%%%%%%%%%%%%%%%%%%%%%%%%%%%%%%%%%%%%%%%%%%%%%%%
\begin{Proposition}
In the above notations, we have
\begin{eqnarray}
J_m^{(-)}(X)=\prod_{Q|m} \frac{(1-X^{f_Q \deg Q})^{g_Q}}
{(1-X^{f_Q^+\deg Q})^{g_Q^+}}
\end{eqnarray}
where $f_Q, f_Q^+$ are the residue class degrees of 
$Q$ in $K_m/k, K_m^+/k$ respectively, and
$g_Q, g_Q^+$ are the numbers of primes in 
$K_m, K_m^+$ respectively over $Q$. 
\end{Proposition}
%%%%%%%%%%%%%%%%%%%%%%%%%%%%%%%%%%%%%%%%%%%%%%%%%%%%%%%%%%%%%%%%%%%%%%%%%%%%
%%
%%proof of proposition 3.1
%%
%%%%%%%%%%%%%%%%%%%%%%%%%%%%%%%%%%%%%%%%%%%%%%%%%%%%%%%%%%%%%%%%%%%%%%%%%%%
\begin{proof}
Notice that $X_m,\;X_m^+$ associate to the $m$-th cyclotomic function
field $K_m$, and its maximal real subfield $K_m^+$ respectively.
Let $Q$ be an irreducible monic polynomial dividing $m$. Put
%%%%%%%%%%%%%%%%%%%%%%%%%%%%%%%%%%%%%%%%%%%%%%%%%%%%%%%%%%%%%%%%%%%%%%%%%%%%
%%
%%equation **
%%
%%%%%%%%%%%%%%%%%%%%%%%%%%%%%%%%%%%%%%%%%%%%%%%%%%%%%%%%%%%%%%%%%%%%%%%%%%%
\begin{eqnarray*}
Y_Q=\{ \; \chi \in X_m \; | \; \chi(Q) \neq 0 \;\}, \;\;
Z_Q=\{ \; \chi \in X_m \; | \; \chi(Q) =1 \;\}.
\end{eqnarray*}
From Theorem 2.1, we have
%%%%%%%%%%%%%%%%%%%%%%%%%%%%%%%%%%%%%%%%%%%%%%%%%%%%%%%%%%%%%%%%%%%%%%%%%%%%
%%
%%equation **
%%
%%%%%%%%%%%%%%%%%%%%%%%%%%%%%%%%%%%%%%%%%%%%%%%%%%%%%%%%%%%%%%%%%%%%%%%%%%%
\begin{eqnarray*}
\prod_{\chi \in X_m}(1-\chi(Q)X^{\deg Q})
&=& \prod_{\chi \in Y_Q}
\bigl(1-\chi(Q)X^{\deg Q}\bigr)\\
&=& \prod_{\chi \in Y_Q/Z_Q}\prod_{\psi \in Z_Q}
\bigl(1-\chi\psi(Q)X^{\deg Q}\bigr)\\
&=& \Bigl(\prod_{\chi \in Y_Q/Z_Q}\bigl(1-\chi(Q)X^{\deg Q}\bigr)
\Bigr)^{g_Q}.
\end{eqnarray*}
Since $Y_Q/Z_Q$ is a cyclic group of order $f_Q$, we have
%%%%%%%%%%%%%%%%%%%%%%%%%%%%%%%%%%%%%%%%%%%%%%%%%%%%%%%%%%%%%%%%%%%%%%%%%%%%
%%
%%equation **
%%
%%%%%%%%%%%%%%%%%%%%%%%%%%%%%%%%%%%%%%%%%%%%%%%%%%%%%%%%%%%%%%%%%%%%%%%%%%%
\begin{eqnarray*}
\prod_{\chi \in Y_Q/Z_Q}
\bigl(1-\chi(Q)X^{\deg Q}\bigr)=\bigl(1-X^{f_Q \deg Q}\bigr).
\end{eqnarray*}
Hence we obtain 
%%%%%%%%%%%%%%%%%%%%%%%%%%%%%%%%%%%%%%%%%%%%%%%%%%%%%%%%%%%%%%%%%%%%%%%%%%%%
%%
%%equation 16
%%
%%%%%%%%%%%%%%%%%%%%%%%%%%%%%%%%%%%%%%%%%%%%%%%%%%%%%%%%%%%%%%%%%%%%%%%%%%%
\begin{eqnarray}
\prod_{\chi \in X_m}
\bigl(1-\chi(Q)X^{\deg Q}\bigr)=\bigl(1-X^{f_Q \deg Q}\bigr)^{g_Q}.
\end{eqnarray}
By the same argument, we have
%%%%%%%%%%%%%%%%%%%%%%%%%%%%%%%%%%%%%%%%%%%%%%%%%%%%%%%%%%%%%%%%%%%%%%%%%%%%
%%
%%equation 17
%%
%%%%%%%%%%%%%%%%%%%%%%%%%%%%%%%%%%%%%%%%%%%%%%%%%%%%%%%%%%%%%%%%%%%%%%%%%%%
\begin{eqnarray}
\prod_{\chi \in X_m^+}
\bigl(1-\chi(Q)X^{\deg Q}\bigr)=\bigl(1-X^{f_Q^+ \deg Q}\bigr)^{g_Q^+}.
\end{eqnarray}
Noting that
$X_m^-=X_m-X_m^+$, we can get the proposition 
from the above equations (16), (17). 
\end{proof}
%%%%%%%%%%%%%%%%%%%%%%%%%%%%%%%%%%%%%%%%%%%%%%%%%%%%%%%%%%%%%%%%%%%%%%%%%%%%
%%
%%
%%
%% comment of proposition 3.1
%%
%%
%%
%%%%%%%%%%%%%%%%%%%%%%%%%%%%%%%%%%%%%%%%%%%%%%%%%%%%%%%%%%%%%%%%%%%%%%%%%%%
There are several consequences of this proposition.
First of all, by Proposition 3.1, we see that
$J_m^{(-)}(X)$ is a polynomial with integral coefficients.
Secondly, if $m$ is the power of an irreducible polynomial $P$, 
the prime $P$ is totally ramified in $K_m/k$ (cf. [Ro 2]).
Hence we obtain $J_m^{(-)}(X)=1$ in this case. 

The next theorem is our main result
of the present paper.
%%%%%%%%%%%%%%%%%%%%%%%%%%%%%%%%%%%%%%%%%%%%%%%%%%%%%%%%%%%%%%%%%%%%%%%%%%%%
%%
%%
%%
%% Theorem  3.1
%%
%%
%%
%%%%%%%%%%%%%%%%%%%%%%%%%%%%%%%%%%%%%%%%%%%%%%%%%%%%%%%%%%%%%%%%%%%%%%%%%%%
\begin{theorem}
Let $m \in A$ be a monic polynomial.
Then, we have
%%%%%%%%%%%%%%%%%%%%%%%%%%%%%%%%%%%%%%%%%%%%%%%%%%%%%%%%%%%%%%%%%%%%%%%%%%%%
%%
%%equation **
%%
%%%%%%%%%%%%%%%%%%%%%%%%%%%%%%%%%%%%%%%%%%%%%%%%%%%%%%%%%%%%%%%%%%%%%%%%%%%
\begin{eqnarray}
\det D_m^{(-)}(X)=P_m^{(-)}(X)J_m^{(-)}(X).
\end{eqnarray}
\end{theorem}
%%%%%%%%%%%%%%%%%%%%%%%%%%%%%%%%%%%%%%%%%%%%%%%%%%%%%%%%%%%%%%%%%%%%%%%%%%%%
%%
%%
%%
%% Proof of Theorem  3.1
%%
%%
%%
%%%%%%%%%%%%%%%%%%%%%%%%%%%%%%%%%%%%%%%%%%%%%%%%%%%%%%%%%%%%%%%%%%%%%%%%%%%
\begin{proof}
For any $\chi \in  X_m$, 
let the monic polynomial $f_{\chi}$ be the conductor of $\chi$.
Define $\tilde{\chi}$ by
%%%%%%%%%%%%%%%%%%%%%%%%%%%%%%%%%%%%%%%%%%%%%%%%%%%%%%%%%%%%%%%%%%%%%%%%%%%
%%
%%equation **
%%
%%%%%%%%%%%%%%%%%%%%%%%%%%%%%%%%%%%%%%%%%%%%%%%%%%%%%%%%%%%%%%%%%%%%%%%%%%%
\begin{eqnarray*}
\tilde{\chi}=\chi \circ \pi_{\chi}
\end{eqnarray*}
where
$\pi_{\chi}: (A/(m))^{\times}\rightarrow (A/(f_{\chi}))^{\times}$
is the natural homomorphism.
Then, we have
%%%%%%%%%%%%%%%%%%%%%%%%%%%%%%%%%%%%%%%%%%%%%%%%%%%%%%%%%%%%%%%%%%%%%%%%%%%
%%
%%equation 19
%%
%%%%%%%%%%%%%%%%%%%%%%%%%%%%%%%%%%%%%%%%%%%%%%%%%%%%%%%%%%%%%%%%%%%%%%%%%%%
\begin{eqnarray}
L(s,\tilde{\chi})=L(s,\chi)\cdot\prod_{Q|m}
\bigl(1-\chi(Q)q^{-s\deg Q}\bigr).
\end{eqnarray}
Fix a non-trivial character $\lambda$ of $\mathbb{F}_q^{\times}$, and
$\psi \in X_m^-\;(\psi |_{\mathbb{F}_q^{\times}}=\lambda)$.
Then we have
%%%%%%%%%%%%%%%%%%%%%%%%%%%%%%%%%%%%%%%%%%%%%%%%%%%%%%%%%%%%%%%%%%%%%%%%%%%%
%%
%%equation **
%%
%%%%%%%%%%%%%%%%%%%%%%%%%%%%%%%%%%%%%%%%%%%%%%%%%%%%%%%%%%%%%%%%%%%%%%%%%%%
\begin{eqnarray*}
\psi \cdot X_m^+=
\{ \chi \in X_m^-\;| \; \chi |_{\mathbb{F}^{\times}_q}=\lambda \}.
\end{eqnarray*}
For a character $\chi \in X_m^-\;(\chi |_{\mathbb{F}_q^{\times}}=\lambda)$, 
there is a unique character
$\phi \in X_m^+$ with $\chi = \psi \cdot \phi$. By the same
argument as in Lemma 3 in [G-R],
%%%%%%%%%%%%%%%%%%%%%%%%%%%%%%%%%%%%%%%%%%%%%%%%%%%%%%%%%%%%%%%%%%%%%%%%%%%%
%%
%%equation **
%%
%%%%%%%%%%%%%%%%%%%%%%%%%%%%%%%%%%%%%%%%%%%%%%%%%%%%%%%%%%%%%%%%%%%%%%%%%%%
\begin{eqnarray*}
L(s,\tilde{\chi})
&=& \sum_{i=1}^{N_m}\tilde{\chi}(\alpha_i)q^{-\text{Deg} (\alpha_i) s}\\
&=& \sum_{i=1}^{N_m}\tilde{\phi}(\alpha_i)
\tilde{\psi}(\alpha_i)c^{\lambda}(\alpha_i)q^{-\text{Deg} (\alpha_i) s}.
\end{eqnarray*}
Notice that $\tilde{\psi}(\alpha)c^{\lambda}(\alpha)$ and $\text{Deg}$ are
functions over $\mathcal{R}_m$, and $\tilde{\phi}$ 
runs through all characters of $\mathcal{R}_m$ when
$\phi$ runs through all characters of $X_m^+$. By 
the Frobenius determinant formula (cf. [Wa], Lemma 5.26),
%%%%%%%%%%%%%%%%%%%%%%%%%%%%%%%%%%%%%%%%%%%%%%%%%%%%%%%%%%%%%%%%%%%%%%%%%%%%
%%
%%equation **
%%
%%%%%%%%%%%%%%%%%%%%%%%%%%%%%%%%%%%%%%%%%%%%%%%%%%%%%%%%%%%%%%%%%%%%%%%%%%%
\begin{eqnarray*}
\prod_{ \chi |_{\mathbb{F}_q^{\times}}=\lambda} L(s,\tilde{\chi})&=&
\prod_{ \phi \in X_m^+}
\sum_{i=1}^{N_m}\tilde{\phi}(\alpha_i)
\tilde{\psi}(\alpha_i)c^{\lambda}(\alpha_i)q^{-\text{Deg} (\alpha_i) s}\\
&=&
\det (\psi(\alpha_i \alpha_j^{-1} ) c_{ij}^{\lambda}
q^{-sd_{ij}})_{i,j=1,2,\cdots N_m}\\
&=& \det D_m^{( \lambda )}(q^{-s}).
\end{eqnarray*}
From the decomposition
%%%%%%%%%%%%%%%%%%%%%%%%%%%%%%%%%%%%%%%%%%%%%%%%%%%%%%%%%%%%%%%%%%%%%%%%%%%%
%%
%%equation **
%%
%%%%%%%%%%%%%%%%%%%%%%%%%%%%%%%%%%%%%%%%%%%%%%%%%%%%%%%%%%%%%%%%%%%%%%%%%%%
\begin{eqnarray*}
X_m^{-}=\bigcup_{\lambda \neq 1}\{\chi \in X_m
\;| \;\chi |_{\mathbb{F}_q^{\times}}=\lambda \},
\end{eqnarray*}
we have
%%%%%%%%%%%%%%%%%%%%%%%%%%%%%%%%%%%%%%%%%%%%%%%%%%%%%%%%%%%%%%%%%%%%%%%%%%%%
%%
%%equation **
%%
%%%%%%%%%%%%%%%%%%%%%%%%%%%%%%%%%%%%%%%%%%%%%%%%%%%%%%%%%%%%%%%%%%%%%%%%%%%
\begin{eqnarray*}
\det D_m^{(-)}(q^{-s})=
\prod_{ \chi \in X_m^-} L(s,\chi) \cdot J_m^{(-)}(q^{-s}).
\end{eqnarray*}
By equation (12), we obtain
%%%%%%%%%%%%%%%%%%%%%%%%%%%%%%%%%%%%%%%%%%%%%%%%%%%%%%%%%%%%%%%%%%%%%%%%%%%%
%%
%%equation 20
%%
%%%%%%%%%%%%%%%%%%%%%%%%%%%%%%%%%%%%%%%%%%%%%%%%%%%%%%%%%%%%%%%%%%%%%%%%%%%
\begin{eqnarray}
\det D_m^{(-)}(q^{-s})=P_m^{(-)}(q^{-s})J_m^{(-)}(q^{-s}).
\end{eqnarray}
Putting $X=q^{-s}$, we obtain the desired result.
\end{proof}
%%%%%%%%%%%%%%%%%%%%%%%%%%%%%%%%%%%%%%%%%%%%%%%%%%%%%%%%%%%%%%%%%%%%%%%%%%%%
%%
%%
%%
%% comment of main theorem 
%%
%%
%%%%%%%%%%%%%%%%%%%%%%%%%%%%%%%%%%%%%%%%%%%%%%%%%%%%%%%%%%%%%%%%%%%%%%%%%%%
We give two remarks of this theorem. To begin with, 
$P_m^{(-)}(X)=1$ when 
$m$ is the monic polynomial of degree $1$.
In fact, we calculate $D_m^{(-)}(X)=1$ in this case. 
Secondly, recall $J_m^{(-)}(X)=1$ when $m$ is the power of
an irreducible polynomial.
Hence $D_m^{(-)}(X)=P_m^{(-)}(X)$ in this case.

As a special case of our result, we obtain
the following determinant formula for relative class numbers.
%%%%%%%%%%%%%%%%%%%%%%%%%%%%%%%%%%%%%%%%%%%%%%%%%%%%%%%%%%%%%%%%%%%%%%%%%%%%
%%
%%
%%
%%Corollary 3.1
%%
%%
%%
%%%%%%%%%%%%%%%%%%%%%%%%%%%%%%%%%%%%%%%%%%%%%%%%%%%%%%%%%%%%%%%%%%%%%%%%%%%
\begin{Corollary}\text{\rm (cf.  [B-K], [A-C-J])}
Let $h_m^-$ be the relative class number of $K_m$.
Put $f_Q^-=f_Q/f_Q^+$ and $g_Q^-=g_Q/g_Q^+$, then
%%%%%%%%%%%%%%%%%%%%%%%%%%%%%%%%%%%%%%%%%%%%%%%%%%%%%%%%%%%%%%%%%%%%%%%%%%%%
%%
%%equation **
%%
%%%%%%%%%%%%%%%%%%%%%%%%%%%%%%%%%%%%%%%%%%%%%%%%%%%%%%%%%%%%%%%%%%%%%%%%%%%
\begin{eqnarray}
\prod_{\lambda \neq 1}\det (c_{ij}^{\lambda})_{i,j=1,2,...,N_m}=W_m^- \cdot h_m^-
\end{eqnarray}
where
%%%%%%%%%%%%%%%%%%%%%%%%%%%%%%%%%%%%%%%%%%%%%%%%%%%%%%%%%%%%%%%%%%%%%%%%%%%%
%%
%%equation **
%%
%%%%%%%%%%%%%%%%%%%%%%%%%%%%%%%%%%%%%%%%%%%%%%%%%%%%%%%%%%%%%%%%%%%%%%%%%%%
\begin{eqnarray}
W_m^-=
\left\{
\begin{array}{ll}
\prod_{Q|m}{(f_Q^-)}^{g_Q^+}  & \mbox{if $g_Q^-=1$ for every prime $Q$ dividing $m$,} \\
0   & \mbox{otherwise.}
\end{array}
\right.
\end{eqnarray}
\end{Corollary}
%%%%%%%%%%%%%%%%%%%%%%%%%%%%%%%%%%%%%%%%%%%%%%%%%%%%%%%%%%%%%%%%%%%%%%%%%%%%
%%
%%
%%
%%Proof of Corollary **
%%
%%
%%
%%%%%%%%%%%%%%%%%%%%%%%%%%%%%%%%%%%%%%%%%%%%%%%%%%%%%%%%%%%%%%%%%%%%%%%%%%%
\begin{proof}
Putting $X=1$ in Theorem 3.1, we see that  
%%%%%%%%%%%%%%%%%%%%%%%%%%%%%%%%%%%%%%%%%%%%%%%%%%%%%%%%%%%%%%%%%%%%%%%%%%%%
%%
%%equation 23
%%
%%%%%%%%%%%%%%%%%%%%%%%%%%%%%%%%%%%%%%%%%%%%%%%%%%%%%%%%%%%%%%%%%%%%%%
\begin{eqnarray}
\det D_m^{(-)}(1)=\prod_{\lambda \neq 1}\det (c_{ij}^{\lambda}),
\end{eqnarray}
and $J_m^{(-)}(1)=W_m^-$ by Proposition 3.1. 
Since $P_m^{(-)}(1)=h_m^-$, we obtain the desired result.
\end{proof}
If $m$ is the power of an irreducible polynomial, 
we see that $W_m^-=1$.
In other case, any prime in $K_m^+$ except infinite primes
is not ramified in $K_m/K_m^+$.
Thus we see $f_Q^-=q-1$ for the prime $Q$ with $g_Q^-=1$.
%%%%%%%%%%%%%%%%%%%%%%%%%%%%%%%%%%%%%%%%%%%%%%%%%%%%%%%%%%%%%%%%%%%%%%%%%%%%
%%
%%
%%
%%
%%
%%
%%
%%
%%
%% section 4
%%
%%
%%
%%
%%
%%
%%
%%
%%%%%%%%%%%%%%%%%%%%%%%%%%%%%%%%%%%%%%%%%%%%%%%%%%%%%%%%%%%%%%%%%%%%%%%%%%%%%
\section{Some coefficients of low degree terms of  $\det D_m^{(-)}(X)$}
In this section, we will calculate the coefficients
of $\det D_m^{(-)}(X)$ of degree $1$, $2$
by using the derivative of determinant.

Let $m \in A$ be a monic polynomial. Noting $\det D_m^{(-)}(0)=1$, we see that
$\det D_m^{(-)}(X)$ can be written by
%%%%%%%%%%%%%%%%%%%%%%%%%%%%%%%%%%%%%%%%%%%%%%%%%%%%%%%%%%%%%%%%%%%%%%%%%%%%
%%
%%equation  24
%%
%%%%%%%%%%%%%%%%%%%%%%%%%%%%%%%%%%%%%%%%%%%%%%%%%%%%%%%%%%%%%%%%%%%%%%%%%%%%
\begin{eqnarray}
\det D_m^{(-)}(X)=1+a_1X+a_2X^2+\cdots
\end{eqnarray}
where $a_i\;(i=1,2,...)$ are integers.
%%%%%%%%%%%%%%%%%%%%%%%%%%%%%%%%%%%%%%%%%%%%%%%%%%%%%%%%%%%%%%%%%%%%%%%%%%%%
%%
%%Proposition 4.1
%%
%%%%%%%%%%%%%%%%%%%%%%%%%%%%%%%%%%%%%%%%%%%%%%%%%%%%%%%%%%%%%%%%%%%%%%%%%%%%%
\begin{Proposition}
Let $m \in A$ be a monic polynomial of degree $d\;(>1)$. Then, we have
%%%%%%%%%%%%%%%%%%%%%%%%%%%%%%%%%%%%%%%%%%%%%%%%%%%%%%%%%%%%%%%%%%%%%%%%%%%%
%%
%%equation  25
%%
%%%%%%%%%%%%%%%%%%%%%%%%%%%%%%%%%%%%%%%%%%%%%%%%%%%%%%%%%%%%%%%%%%%%%%%%%%%%%
\begin{eqnarray}
(1)\; a_1&=& 0,\\
(2)\; a_2&=& 0 \;\;\; (\text{if} \; \deg m>2),\\
(3)\; a_2&=& \frac{N_m}{2}\{(q-1)(1-C_m)+N_m-1\}  \;\;\;(\text{if} \;\deg m=2),
\end{eqnarray}
where 
%%%%%%%%%%%%%%%%%%%%%%%%%%%%%%%%%%%%%%%%%%%%%%%%%%%%%%%%%%%%%%%%%%%%%%%%%%%%
%%
%%equation  28
%%
%%%%%%%%%%%%%%%%%%%%%%%%%%%%%%%%%%%%%%%%%%%%%%%%%%%%%%%%%%%%%%%%%%%%%%%%%%%%%
\begin{eqnarray}
C_m={}^{\#}\{i=1,2,..., N_m\; |\;L(\alpha_i^{-1})=1\}.
\end{eqnarray}
Here ${}^{\#}A$ is the number of elements of a set $A$.
\end{Proposition}
By Proposition 3.1, we can obtain $J_m^{(-)}(X)$. Hence 
we can also calculate coefficients of low degree terms of $P_m^{(-)}(X)$.\\ 
To prove Proposition 3.1, we first state the next lemma, which
can be shown by simple calculations.
%%%%%%%%%%%%%%%%%%%%%%%%%%%%%%%%%%%%%%%%%%%%%%%%%%%%%%%%%%%%%%%%%%%%%%%%%%%%
%%
%%Lemma  **
%%
%%%%%%%%%%%%%%%%%%%%%%%%%%%%%%%%%%%%%%%%%%%%%%%%%%%%%%%%%%%%%%%%%%%%
\begin{Lemma}
Let $F(X)=(f_{ij}(X))_{i,j}$ be a matrix with one variable.
If $F(X)$ is twice differentiable and invertible at $X=X_0$, then
%%%%%%%%%%%%%%%%%%%%%%%%%%%%%%%%%%%%%%%%%%%%%%%%%%%%%%%%%%%%%%%%%%%%%%%%%%%%
%%
%%equation  **
%%
%%%%%%%%%%%%%%%%%%%%%%%%%%%%%%%%%%%%%%%%%%%%%%%%%%%%%%%%%%%%%%%%%%%%
\begin{eqnarray*}
\text{$(1)$ }\frac{d^{\;} \det F(X)}{dX}\bigg|_{X=X_0}&=& \det F(X_0)
\cdot
\text{\rm Tr}\Bigl(F(X_0)^{-1}\frac{d F}{dX}(X_0)\Bigr),\\
\text{$(2)$ }
\frac{ d^{2} \det F(X) }{ d{X^2} }\bigg|_{X=X_0} &=&
 \det F(X_0)\cdot
\Bigl\{\;
\text{\rm Tr}\Bigl( F(X_0)^{-1} \frac{d^{2} F}{d X^{2}}(X_0) \Bigr)-\\
&&\text{\rm Tr}\Bigl( F(X_0)^{-1} \frac{d F}{dX} (X_0)
   F(X_0)^{-1} \frac{dF}{dX} (X_0) \Bigr)+\\
&&\text{\rm Tr}\Bigl( F(X_0)^{-1} \frac{dF}{dX} (X_0) \Bigr)^2
\;\Bigr\},
\end{eqnarray*}
where $\text{\rm Tr}(A)$ is the trace of the matrix $A$.
\end{Lemma}
Now we prove the proposition.
%%%%%%%%%%%%%%%%%%%%%%%%%%%%%%%%%%%%%%%%%%%%%%%%%%%%%%%%%%%%%%%%%%%%%%%%%%%%
%%
%%
%%Proof of Proposition **
%%
%%
%%%%%%%%%%%%%%%%%%%%%%%%%%%%%%%%%%%%%%%%%%%%%%%%%%%%%%%%%%%%%%%%%%%%
\begin{proof}
Let $\lambda$ be a non-trivial character of $\mathbb{F}_q^{\times}$,
and write 
%%%%%%%%%%%%%%%%%%%%%%%%%%%%%%%%%%%%%%%%%%%%%%%%%%%%%%%%%%%%%%%%%%%%%%%%%%%%
%%
%% equation **
%%
%%%%%%%%%%%%%%%%%%%%%%%%%%%%%%%%%%%%%%%%%%%%%%%%%%%%%%%%%%%%%%%%%%%%
\begin{eqnarray*}
\det D_{m}^{(\lambda)}(X)=1+a_1^{\lambda}X+a_2^{\lambda}X^2+\cdots.
\end{eqnarray*}
Notice that $D_{m}^{(\lambda)}(0)$ is the unit matrix and
%%%%%%%%%%%%%%%%%%%%%%%%%%%%%%%%%%%%%%%%%%%%%%%%%%%%%%%%%%%%%%%%%%%%%%%%%%%%
%%
%% equation 29
%%
%%%%%%%%%%%%%%%%%%%%%%%%%%%%%%%%%%%%%%%%%%%%%%%%%%%%%%%%%%%%%%%%%%%
\begin{eqnarray}
\frac{d D_{m}^{( \lambda )}}{dX}(0)=(l_{ij})_{i,j=1,2,\cdots,N_m}
\end{eqnarray}
where 
%%%%%%%%%%%%%%%%%%%%%%%%%%%%%%%%%%%%%%%%%%%%%%%%%%%%%%%%%%%%%%%%%%%%%%%%%%%%
%%
%% equation 30
%%
%%%%%%%%%%%%%%%%%%%%%%%%%%%%%%%%%%%%%%%%%%%%%%%%%%%%%%%%%%%%%%%%%%%
\begin{eqnarray}
l_{ij}=
\left\{
\begin{array}{rl}
0  & \mbox{if $d_{ij}=0$ or $d_{ij}>1$}, \\
c_{ij}^{\lambda}  & \mbox{if $d_{ij}$ = 1 }.
\end{array}
\right.
\end{eqnarray}
By Lemma 4.1, $a_1^{\lambda}=0$ and
%%%%%%%%%%%%%%%%%%%%%%%%%%%%%%%%%%%%%%%%%%%%%%%%%%%%%%%%%%%%%%%%%%%%%%%%%%%%
%%
%% equation **
%%
%%%%%%%%%%%%%%%%%%%%%%%%%%%%%%%%%%%%%%%%%%%%%%%%%%%%%%%%%%%%%%%%%%%
\begin{eqnarray*}
a_2^{\lambda}=-\frac{1}{2}\text{\rm Tr}
\Bigl(  \Bigl( \;\frac{dD_m^{( \lambda )}}{dX}(0)\;  \Bigr)^2 \Bigr).
\end{eqnarray*}
Thus we have assertion (1).
If $\deg m>2$, there is no combination $(i,j)$ such that $d_{ij}=1,$
and $d_{ji}=1$. Thus we have $a_2^{\lambda}=0$ in the case $\deg m>2$.
Since $a_2=\sum_{\lambda \neq 1}a_2^{\lambda}$, we obtain assertions (2).

Next we prove the case when $\deg m=2$. In this case, we have
%%%%%%%%%%%%%%%%%%%%%%%%%%%%%%%%%%%%%%%%%%%%%%%%%%%%%%%%%%%%%%%%%%%%%%%%%%%%
%%
%% equation 31
%%
%%%%%%%%%%%%%%%%%%%%%%%%%%%%%%%%%%%%%%%%%%%%%%%%%%%%%%%%%%%%%%%%%%%
\begin{eqnarray}
l_{ij}=
\left\{
\begin{array}{rl}
0 & \mbox{if $i=j$}, \\
c_{ij}^{\lambda} & \mbox{if $i\neq j$}.
\end{array}
\right.
\end{eqnarray}
Thus we have
%%%%%%%%%%%%%%%%%%%%%%%%%%%%%%%%%%%%%%%%%%%%%%%%%%%%%%%%%%%%%%%%%%%%%%%%%%%%
%%
%% equation **
%%
%%%%%%%%%%%%%%%%%%%%%%%%%%%%%%%%%%%%%%%%%%%%%%%%%%%%%%%%%%%%%%%%%%%
\begin{eqnarray*}
\sum_{\lambda \neq 1}a_2^{\lambda} &=&
\sum_{\lambda \neq 1}\Bigl(\frac{N_m}{2}-\frac{1}{2}
\sum_{i=1}^{N_m}\sum_{j=1}^{N_m}
\lambda^{-1}(L(\alpha_i\alpha_j^{-1})L(\alpha_j\alpha_i^{-1}))\Bigr)\\
&=& \frac{N_m(q-2)}{2}-\frac{1}{2}\sum_{i=1}^{N_m}\sum_{j=1}^{N_m}e_{ij}
\end{eqnarray*}
where
%%%%%%%%%%%%%%%%%%%%%%%%%%%%%%%%%%%%%%%%%%%%%%%%%%%%%%%%%%%%%%%%%%%%%%%%%%%%
%%
%% equation **
%%
%%%%%%%%%%%%%%%%%%%%%%%%%%%%%%%%%%%%%%%%%%%%%%%%%%%%%%%%%%%%%%%%%%%
\begin{eqnarray*}
e_{ij}=
\left\{
\begin{array}{rl}
q-2  & \mbox{ if $L(\alpha_i\alpha_j^{-1}) L(\alpha_j\alpha_i^{-1})=1$}, \\
-1      & \mbox{otherwise}.
\end{array}
\right.
\end{eqnarray*}
For any $i,j \in \{1,2,\cdots N_m \}$, 
there are $\gamma_{ij} \in \mathbb{F}_q^{\times}$ and
$\beta_{ij} \in (A/(m))^{\times}$ with $L(\beta_{ij})=1$
such that $\alpha_i\alpha_j^{-1}=\gamma_{ij}\beta_{ij}$.
Then we have
%%%%%%%%%%%%%%%%%%%%%%%%%%%%%%%%%%%%%%%%%%%%%%%%%%%%%%%%%%%%%%%%%%%%%%%%%%%%
%%
%% equation **
%%
%%%%%%%%%%%%%%%%%%%%%%%%%%%%%%%%%%%%%%%%%%%%%%%%%%%%%%%%%%%%%%%%%%%
\begin{eqnarray*}
L(\alpha_i\alpha_j^{-1}) L(\alpha_j\alpha_i^{-1})=L(\beta_{ij}^{-1}).
\end{eqnarray*}
Noting 
%%%%%%%%%%%%%%%%%%%%%%%%%%%%%%%%%%%%%%%%%%%%%%%%%%%%%%%%%%%%%%%%%%%%%%%%%%%%
%%
%% equation **
%%
%%%%%%%%%%%%%%%%%%%%%%%%%%%%%%%%%%%%%%%%%%%%%%%%%%%%%%%%%%%%%%%%%%%
\begin{eqnarray*}
\{\beta_{ij}\;| j=1,2,\cdots N_m\}=\{\alpha_j\; | j=1,2,\cdots N_m\},
\end{eqnarray*}
we have
%%%%%%%%%%%%%%%%%%%%%%%%%%%%%%%%%%%%%%%%%%%%%%%%%%%%%%%%%%%%%%%%%%%%%%%%%%%%
%%
%% equation **
%%
%%%%%%%%%%%%%%%%%%%%%%%%%%%%%%%%%%%%%%%%%%%%%%%%%%%%%%%%%%%%%%%%%%%
\begin{eqnarray*}
\sum_{j=1}^{N_m}e_{ij}=(q-1)C_m-N_m.
\end{eqnarray*}
Thus we have the desired result. 
\end{proof}
We consider the case when $m=T^2+aT+b \in A$.
If $\alpha = T-c$ satisfies $L(\alpha^{-1})=1$,
then $c$ is a root of the equation $T^2+aT+b+1$. Thus
we obtain $C_m \leqq 3$.
%%%%%%%%%%%%%%%%%%%%%%%%%%%%%%%%%%%%%%%%%%%%%%%%%%%%%%%%%%%%%%%%%%%%%%%%%%%%
%%
%%
%%
%%section  5
%%
%%
%%
%%%%%%%%%%%%%%%%%%%%%%%%%%%%%%%%%%%%%%%%%%%%%%%%%%%%%%%%%%%%%%%%%%%%
\section{Examples}
In this section, we give some examples.
%%%%%%%%%%%%%%%%%%%%%%%%%%%%%%%%%%%%%%%%%%%%%%%%%%%%%%%%%%%%%%%%%%%%%%%%%%%%
%%
%% Example 1
%%
%%%%%%%%%%%%%%%%%%%%%%%%%%%%%%%%%%%%%%%%%%%%%%%%%%%%%%%%%%%%
\begin{Example}
For $q=3$ and $m=T^2+1$, we see that the extension degree of $K_m/k$ is $8$, and
$N_m =4$. Since the polynomial 
$m$ is irreducible, we have $\det D_m^{(-)}(X)=P_m^{(-)}(X)$.
Put 
%%%%%%%%%%%%%%%%%%%%%%%%%%%%%%%%%%%%%%%%%%%%%%%%%%%%%%%%%%%%%%%%%%%%%%%%%%%%
%%
%% equation **
%%
%%%%%%%%%%%%%%%%%%%%%%%%%%%%%%%%%%%%%%%%%%%%%%%%%%%%%%%%%%%%%%%%%%%
\begin{eqnarray*}
\alpha_1=1,\; \alpha_2=T, \;\alpha_3=T+1,\; \alpha_4=T+2.
\end{eqnarray*}
Then we have
\begin{eqnarray*}
P_m^{(-)}(X)&=&\det D_{m}^-(X) \\
&=&
\left|
\begin{array}{rrrr}
1   &-X       &X      &X          \\
X   & 1       &-X      &X        \\
X   & -X       & 1      &-X        \\
X   & X      &X       &1
\end{array}
\right| \\
&=& 1-2X^2+9X^4.
\end{eqnarray*}
The relative class number $h_m^-$ of $K_m$ is $P_m^{(-)}(1)=8$.
\end{Example}
%%%%%%%%%%%%%%%%%%%%%%%%%%%%%%%%%%%%%%%%%%%%%%%%%%%%%%%%%%%%%%%%%%%%%%%%%%%%
%%
%% Example  **
%%
%%%%%%%%%%%%%%%%%%%%%%%%%%%%%%%%%%%%%%%%%%%%%%%%%%%%%%%%%%%%
\begin{Example}
For $q=3$ and $m=T^3+T^2$, 
we see that the extension degree of $K_m/k$ is $12$, and
$N_m =6$. 
Put
%%%%%%%%%%%%%%%%%%%%%%%%%%%%%%%%%%%%%%%%%%%%%%%%%%%%%%%%%%%%%%%%%%%%%%%%%%%%
%%
%% equation **
%%
%%%%%%%%%%%%%%%%%%%%%%%%%%%%%%%%%%%%%%%%%%%%%%%%%%%%%%%%%%%%%%%%%%%
\begin{eqnarray*}
&\alpha_1&=1,\; \alpha_2=T^2+2T+2, \;\alpha_3=T^2+T+1,\\
&\alpha_4&=T+2,\;\alpha_5=T^2+1\,\;\alpha_6=T^2+T+2.
\end{eqnarray*}
Then we have 
%%%%%%%%%%%%%%%%%%%%%%%%%%%%%%%%%%%%%%%%%%%%%%%%%%%%%%%%%%%%%%%%%%%%%%%%%%%%
%%
%% equation **
%%
%%%%%%%%%%%%%%%%%%%%%%%%%%%%%%%%%%%%%%%%%%%%%%%%%%%%%%%%%%%%%%%%%%%
\begin{eqnarray*}
\det D_{m}^{(-)}(X) &=&
\left|
\begin{array}{rrrrrr}
1   & X    & -X^2  & X^2  & X^2 & -X^2 \\
X^2 & 1    & -X^2  & -X^2 & -X^2& -X \\
X^2 & X^2  & 1     & X    & -X^2& X^2 \\
X   & X^2  & X^2   & 1    & X^2 & X^2 \\
X^2 & X^2  & -X    & -X^2 & 1   & X^2\\ 
X^2 & -X^2 & -X^2  & X^2  & X   & 1
\end{array}
\right| \\
&=& 1-6 X^3-3X^4-6X^5+23X^6+30X^7+6X^8-18X^9-27X^{10}, 
\end{eqnarray*}
and 
%%%%%%%%%%%%%%%%%%%%%%%%%%%%%%%%%%%%%%%%%%%%%%%%%%%%%%%%%%%%%%%%%%%%%%%%%%%%
%%
%% equation **
%%
%%%%%%%%%%%%%%%%%%%%%%%%%%%%%%%%%%%%%%%%%%%%%%%%%%%%%%%%%%%%%%%%%%%
\begin{eqnarray*}
J_m^{(-)}(X)=1+X-X^3-X^4.
\end{eqnarray*}
Thus we obtain
\begin{eqnarray*}
P_m^{(-)}(X)&=&\frac{\det D_{m}^{(-)}(X)}{J_m^{(-)}(X)}\\
            &=& 1-X+X^2-6X^3+3X^4-9X^5+27X^6.
\end{eqnarray*}
The relative class number $h_m^-$ of $K_m$ is $P_m^{(-)}(1)=16$.
\end{Example}
%%%%%%%%%%%%%%%%%%%%%%%%%%%%%%%%%%%%%%%%%%%%%%%%%%%%%%%%%%%%%%%%%%%%%%%%%%%%
%%
%%  Acknowledgement
%%
%%%%%%%%%%%%%%%%%%%%%%%%%%%%%%%%%%%%%%%%%%%%%%%%%%%%%%%%%%%%%%%%%%%{bf Acknowledgements.}
{\bf Acknowledgements.}
I would like to thank Prof. Kohji Matsumoto for his valuable comments.
%%%%%%%%%%%%%%%%%%%%%%%%%%%%%%%%%%%%%%%%%%%%%%%%%%%%%%%%%%%%%%%%%%%%%%%%%%%
%%
%%bibliography
%%
%%%%%%%%%%%%%%%%%%%%%%%%%%%%%%%%%%%%%%%%%%%%%%%%%%%%%%%%%%%%%%%%%%%%%%%%%%%%%

%%%%%%%%%%%%%%%%%%%%%%%%%%%%%%%%%%%%%%%%%%%%%%%%%%%%%%%%%%%%%%%%%%%
%%
%%
%% Adress
%%
%%
%%
%%%%%%%%%%%%%%%%%%%%%%%%%%%%%%%%%%%%%%%%%%%%%%%%%%%%%%%%%%%%%%%%%%%%%
\text{ }\\
Daisuke Shiomi (JSPS Research  Fellow)\\
Graduate School of Mathematics\\
Nagoya University\\
Chikusa-ku, Nagoya 464-8602\\
Japan\\
Mail: m05019e@math.nagoya-u.ac.jp\\
\end{document}